\documentclass[12pt,twoside,reqno]{amsart}
\usepackage{amsmath}
\usepackage{amsfonts}
\usepackage{amssymb}
\usepackage{color}
\usepackage{mathrsfs}
\usepackage{cite}
\newcommand{\RR}{\mathbb R}

\newcommand{\beqn}{\begin{equation}}
\newcommand{\eeqn}{\end{equation}}
\newcommand{\bean}{\begin{eqnarray}}
\newcommand{\eean}{\end{eqnarray}}
\DeclareMathAlphabet{\mathpzc}{OT1}{pzc}{m}{it}
\newtheorem{theorem}{Theorem}[section]

\newtheorem{lemma}[theorem]{Lemma}
\newtheorem{proposition}[theorem]{Proposition}
\newtheorem{definition}[theorem]{Definition}

\numberwithin{equation}{section}

\newcommand{\id}{\mathop{\rm id}\nolimits}
\begin{document}
\title{Absence of gelation and self-similar behavior for\\ a coagulation-fragmentation equation}
\author{Philippe Lauren\c cot}
\address{Institut de Math\'ematiques de Toulouse, UMR~5219, Universit\'e de Toulouse, CNRS, F--31062 Toulouse Cedex 9, France} 
\email{laurenco@math.univ-toulouse.fr}
\author{Henry van Roessel}
\address{Department of Mathematical and Statistical Sciences, University of Alberta, Edmonton, Alberta, Canada}
\email{henry.vanroessel@ualberta.ca}

\keywords{coagulation, singular fragmentation, mass conservation, self-similarity, Laplace transform, characteristics}
\subjclass{45K05, 44A10, 35L60, 35A30, 82C22}

\date{\today}

\begin{abstract}
The dynamics of a coagulation-fragmentation equation with multiplicative coagulation kernel and critical singular fragmentation is studied. In contrast to the coagulation equation, it is proved that fragmentation prevents the occurrence of the gelation phenomenon and a mass-conserving solution is constructed. The large time behavior of this solution is shown to be described by a self-similar solution. In addition, the second moment is finite for positive times whatever its initial value. The proof relies on the Laplace transform which maps the original equation to a first-order nonlinear hyperbolic equation with a singular source term. A precise study of this equation is then performed with the method of characteristics.
\end{abstract}

\maketitle

%
%
\pagestyle{myheadings}
\markboth{\sc{Ph. Lauren\c cot \& H. van Roessel}}{\sc{Absence of gelation and self-similarity for a coagulation-fragmentation equation}}

\section{Introduction}\label{sec1}

Coagulation-fragmentation equations are mean-field models describing the growth of clusters changing their sizes under the combined effects of (binary) merging and breakup. Denoting the size distribution function of the particles with mass $x>0$ at time $t>0$ by $f=f(t,x)\ge 0$, the coagulation equation with multiple fragmentation reads
\begin{eqnarray}
\partial_t f(t,x) & = & \frac{1}{2} \int_0^x K(x-y,y) f(t,x-y) f(t,y)\ dy - \int_0^\infty K(x,y) f(t,x) f(t,y)\ dy \nonumber \\
& & -\  a(x) f(t,x) + \int_x^\infty a(y) b(x|y) f(t,y)\ dy\ , \label{i1}
\end{eqnarray}
for $(t,x)\in (0,\infty)^2$. In \eqref{i1}, $K$ denotes the coagulation kernel which is a non-negative and symmetric function $K(x,y)=K(y,x)\ge 0$ of $(x,y)$ accounting for the likelihood of two particles with respective masses $x$ and $y$ to merge. Next, $a(x)\ge 0$ denotes the overall rate of breakup of particles of size $x>0$ and $b(x|y)$ is the daughter distribution function describing the probability that the fragmentation of a particle with mass $y$ produces a particle with mass $x\in (0,y)$. Conservation of matter during fragmentation events requires
\begin{equation}
\int_0^y x b(x|y)\ dx = y\ , \qquad y\in (0,\infty)\ , \label{i2}
\end{equation}
and since the total mass of particles is preserved during coagulation events, the total mass is expected to be constant throughout time evolution, that is, 
\begin{equation}
\int_0^\infty x f(t,x)\ dx = \int_0^\infty x f(0,x)\ dx\ , \quad t>0\ . \label{i100}
\end{equation}

The total mass conservation \eqref{i100} is actually a key issue in the analysis of coagulation-fragmentation equations as it might be infringed during time evolution. More precisely, in the absence of fragmentation ($a\equiv 0$), it is by now known that the coagulation kernels $K$ split into two classes according to their growth for large values of $x$ and $y$: either 
\begin{equation}
K(x,y)\le \kappa (2+x+y)\ , \quad (x,y)\in (0,\infty)^2\ , \label{i101}
\end{equation}
for some $\kappa>0$ and solutions to the coagulation equation are expected to be \textit{mass-conserving} and satisfies \eqref{i100}. Or 
\begin{equation}
K(x,y)\ge \kappa (xy)^{\lambda/2}\ , \quad (x,y)\in (0,\infty)^2\ , \label{i102}
\end{equation} 
for some $\lambda>1$ and $\kappa>0$ and conservation of mass breaks down in finite time, a phenomenon commonly referred to as gelation. From a physical viewpoint, it corresponds to a runaway growth of the dynamics, leading to the formation of a particle with infinite mass in finite time and resulting in a loss of matter as \eqref{i1} only accounts for the evolution of particles with finite mass. In other words, there is a finite time $T_{gel}$, the so-called \textit{gelation time}, such that
$$
\int_0^\infty x f(t,x) = \int_0^\infty x f(0,x)\ dx\ , \;\; t\in [0,T_{gel}) \;\;\text{ and }\;\; \int_0^\infty x f(t,x) < \int_0^\infty x f(0,x)\ dx\ , \;\; t>T_{gel}\ .
$$ 
Alternatively,
\begin{equation*}
T_{gel} := \inf{\left\{ t>0\ :\ \int_0^\infty x f(t,x)\ dx <  \int_0^\infty x f(0,x)\ dx \right\}}\ . 
\end{equation*} 
The fact that gelation occurs for coagulation kernels satisfying \eqref{i102} was conjectured at the beginning of the eighties and supported by a few examples of explicit or closed-form solutions to \eqref{i1} \cite{LT81, Le83}. That it indeed takes place for such coagulation kernels and arbitrary initial data was shown twenty years later in \cite{EMP02}. Since fragmentation reduces the sizes of the clusters, it is rather expected to counteract gelation and it has indeed been established that strong fragmentation prevents the occurrence of gelation \cite{dC95} but does not impede the occurrence of gelation if it is too weak \cite{ELMP03, EMP02, Lt00, VZ89}. More specifically, for the following typical choices of coagulation kernel~$K$ and overall breakup rate~$a$
\begin{eqnarray}
K(x,y) & = & \kappa \left( x^\alpha y^\beta + x^\beta y^\alpha \right)\ , \quad (x,y) \in (0,\infty)\times (0,\infty)\ , \label{i3} \\
a(x) & = & k x^\gamma\ , \quad x\in (0,\infty)\ , \nonumber
\end{eqnarray}
with $\alpha \le \beta \le 1$, $\lambda:= \alpha+\beta>1$, $\gamma>0$, $\kappa>0$, and $k>0$, it is conjectured in \cite{VZ89, Pi12} that mass-conserving solutions exist when $\gamma>\lambda-1$ while gelation occurs when $\gamma<\lambda-1$, the critical case $\gamma=\lambda-1$ being more involved and possibly depending as well on the properties of the daughter distribution function $b$. Except for the critical case $\gamma=\lambda-1$, mathematical proofs of these conjectures are provided in \cite{dC95, ELMP03, EMP02} for some specific classes of daughter distribution functions $b$.

\medskip

More generally, due to the possible breakdown of mass conservation, the analysis of the existence of solutions to \eqref{i1} follows two directions since the pioneering works by Melzak \cite{Me57}, McLeod \cite{ML62a,ML62b,ML64}, White \cite{Wh80}, and Spouge \cite{Sp84}. On the one hand, several works have been devoted to the construction of mass-conserving solutions to \eqref{i1} for coagulation kernels satisfying \eqref{i101} under various assumptions on the breakup rate $a$ and the daughter distribution function $b$, see \cite{BC90, BL11, BL12, DS96b, ELMP03, La04, Lt02, MLM97, Me57, Wh80} and the references therein. On the other hand, weak solutions which need not satisfy \eqref{i100} have been constructed for large classes of coagulation kernels $K$, breakup rate $a$, and daughter distribution function $b$, see \cite{Ca06, EW01, ELMP03, EMRR05, GKW11, GLW12, Lt00, No99, St89}. Existence of mass-conserving solutions to \eqref{i1} for coagulation kernels satisfying \eqref{i102} with strong fragmentation has been established in \cite{dC95, ELMP03}. In addition, uniqueness of mass-conserving solutions has been investigated in \cite{dC95, GW11, No99, St90b}. A survey of earlier results and additional references may be found in \cite{Du94b, LM04}. Let us mention here that a common feature of the above mentioned works is that the total number of particles $n(y)$ resulting from the breakup of a cluster of size $y>0$ defined by
\begin{equation}
n(y) := \int_0^y b(x|y)\ dx\ , \label{i104}
\end{equation}
is assumed to be bounded or even constant and thus does not include the case where the breakup of particles could produce infinitely many particles. This possibility shows up in the class of breakup rates $a$ and daughter distribution functions $b$ satisfying \eqref{i2} which are derived in \cite{MZ87} and given by
\begin{equation}
a(x) = k x^\gamma\ , \quad b(x|y) = (\omega+2) \frac{x^\omega}{y^{\omega+1}}\ , \quad 0<x<y\ , \label{i4}
\end{equation} 
with $k>0$, $\gamma\in\mathbb{R}$, and $\omega\in (-2,0]$. Clearly, the total number of particles $n(y)$ resulting from the breakup of a particle of size $y>0$ defined in \eqref{i104} is finite if $\omega\in (-1,0]$ and infinite if $\omega\in (-2,-1]$, the latter being then excluded from the already available studies.

\medskip

The purpose of this note is to contribute to the analysis of the critical case when the homogeneity indices $\lambda = \alpha+\beta$ and $\gamma$ of $K$ in \eqref{i3} and $a$ in \eqref{i4} are related by $\gamma=\lambda-1$ in the particular case
\begin{equation}
\alpha=\beta=1\ , \quad \kappa = \frac{1}{2} \ , \quad \gamma=1\ , \quad \omega=-1\ , \quad k>0 \ . \label{i5}
\end{equation}
Note that the choice of $\omega$ in \eqref{i5} corresponds to the production of an infinite number of particles during each fragmentation event. With this choice of parameters, equation~\eqref{i1} reads
\begin{eqnarray}
\partial_t f(t,x) & = & \frac{1}{2} \int_0^x (x-y)y f(t,x-y) f(t,y)\ dy - \int_0^\infty xy f(t,x) f(t,y)\ dy \nonumber \\
& & -\  k x f(t,x) + k \int_x^\infty \frac{y}{x} f(t,y)\ dy\ . \label{i6}
\end{eqnarray}
At this point we realize that, introducing $\nu(t,x) := x f(t,x)$ and multiplying \eqref{i6} by $x$, an alternative formulation of \eqref{i6} reads 
\begin{eqnarray}
\partial_t \nu(t,x) & = & \frac{x}{2} \int_0^x \nu(t,x-y) \nu(t,y)\ dy - x \int_0^\infty \nu(t,x) \nu(t,y)\ dy \nonumber \\
& & -\  k x \nu(t,x) + k \int_x^\infty \nu(t,y)\ dy\ . \label{i7}
\end{eqnarray} 
One advantage of this formulation is that it cancels the possible singularity of $f(t,x)$ as $x\to 0$. More precisely, while $\nu(t)$ is expected to be integrable due to the boundedness of the total mass, it is unclear whether $f(t)$ is integrable for positive times, even if the initial total number of particles is finite. Indeed, recall that, for the fragmentation rate under consideration, an infinite number of particles is produced during each fragmentation event. Another advantage of this formulation and the choice of the rate coefficients $K$, $a$, $b$ is that taking the Laplace transform of \eqref{i6} leads us to the following partial differential equation 
\begin{equation}
\partial_t L(t,s) = (L(t,0)+k-L(t,s))\ \partial_s L(t,s) + k\  \frac{L(t,0)-L(t,s)}{s}\ , \quad t>0\ ,\ s>0\ , \label{i8} 
\end{equation}
for the Laplace transform 
$$
L(t,s) := \int_0^\infty \nu(t,x) e^{-sx}\ dx = \int_0^\infty x f(t,x) e^{-sx}\ dx\ , \quad (t,s)\in (0,\infty)^2\ ,
$$
of $\nu$. Note that $L(t,0)$ is the total mass at time $t$ and does not depend on time in the absence of gelation. Our aim is then to study the behavior of the solutions to \eqref{i8} and thereby obtain some information on the behavior of solutions to \eqref{i7} (and thus also on \eqref{i6}). To this end, it turns out that it is more appropriate to work with \textit{measure-valued} solutions to \eqref{i7}. Let $\mathfrak{M}^+$ be the space of non-negative bounded measures on $(0,\infty)$ and fix an initial condition $\nu^{in}$ satisfying
\begin{equation}
\nu^{in} \in \mathfrak{M}^+\ , \quad \int_0^\infty \nu^{in}(dx) =1\ . \label{i9}
\end{equation}
We may actually assume without loss of generality that $\nu^{in}$ is a probability measure after a suitable rescaling. Weak solutions to \eqref{i7} are then defined as follows:

\begin{definition}\label{defws}
Given an initial condition $\nu^{in}$ satisfying \eqref{i9}, a weak solution to \eqref{i7} with initial condition $\nu^{in}$ is a weakly continuous map $\nu: [0,\infty)\to \mathfrak{M}^+$ such that
\begin{eqnarray}
\int_0^\infty \vartheta(x) \nu(t,dx) & = & \int_0^\infty \vartheta(x) \nu^{in}(dx) + \int_0^t \int_0^\infty \int_0^\infty x [\vartheta(x+y) - \vartheta(x)] \nu(\tau,dx) \nu(\tau,dy)\ d\tau \nonumber \\
& & +\ k \int_0^t \int_0^\infty \left[ \int_0^x \vartheta(y)\ dy - x \vartheta(x) \right]\ \nu(\tau,dx)\ d\tau \label{i10}
\end{eqnarray}
for all $t>0$ and $\vartheta\in C([0,\infty))$ with compact support.
\end{definition}

We next denote the subset of non-negative bounded measures on $(0,\infty)$ with finite first moment by $\mathfrak{M}_1^+$. We now state the main result:

\begin{theorem}\label{thint0}
Let $\nu^{in}$ be an initial condition satisfying \eqref{i9}. There is a weak solution $\nu$ to \eqref{i7} in the sense of Definition~\ref{defws} which is mass-conserving for all times, that is, 
\begin{equation*}
\int_0^\infty \nu(t,dx) = 1 = \int_0^\infty \nu^{in}(dx)\ , \quad t>0\ . 
\end{equation*}
Moreover, $\nu(t)\in \mathfrak{M}^+_1$ for all $t>0$ and
\begin{equation}
\int_0^\infty x \nu(t,dx) \sim \frac{e^{1/k}-1}{t} \;\;\text{ as }\;\; t\to \infty\ . \label{i12}
\end{equation}
Furthermore, introducing
$$
M(t,x) := \int_0^x \nu(t,dy)\ , \quad (t,x)\in (0,\infty)\times (0,\infty)\ , 
$$ 
there is a probability measure $\nu_\star\in\mathfrak{M}^+$ such that
\begin{equation}
\lim_{t\to\infty} M\left( t , \frac{x}{t} \right) = M_\star(x) := \int_0^x \nu_\star(dy) \;\;\text{ for all }\;\; x\in (0,\infty)\ . \label{i13}
\end{equation} 
\end{theorem}

More information is actually available on $\nu_\star$. Indeed, it follows from Proposition~\ref{prltb0} that the Laplace transform $L_\star$ of $\nu_\star$ is given by
\begin{equation}
L_\star(s) := 1 + s - k W\left( \frac{s}{k} e^{(1+s)/k} \right)\ , \quad s\ge 0\ , \label{i14}
\end{equation}
where $W$ is the so-called Lambert $W$-function, that is, the inverse function of $z\mapsto z e^z$ in $(0,\infty)$. 

Recalling that the gelation time is finite for all solutions to the coagulation equation \eqref{i1} in the absence of fragmentation ($k=0$), we deduce from Theorem~\ref{thint0} that adding fragmentation prevents the occurrence of gelation whatever the value of $k>0$. The value of the parameter $k$ thus plays only a minor role in that direction but it comes into play in the large time behavior as the self-similar profile \eqref{i14} depends explicitly on $k$. That a self-similar behavior for large times is plausible for rate coefficients $K$ and $a$ given by \eqref{i3} and \eqref{i4} and satisfying $\gamma=\alpha+\beta-1$ is expected from the scaling analysis performed in \cite{Pi12,VZ89}, and we show in Theorem~\ref{thint0} that this is indeed the case for the particular choice \eqref{i5}. Moreover, the convergence to zero of the first moment of $\nu$ (which corresponds to the second moment of $f$) as $t\to \infty$ gives a positive answer to a conjecture in \cite{VZ89}, providing in addition an optimal rate of convergence to zero.

\medskip

As already mentioned, the proof of Theorem~\ref{thint0} relies on the study of the Laplace transform $L$ of solutions $\nu$ to \eqref{i7} which are related to solutions $f$ of \eqref{i6} by $\nu(t,x)=x f(t,x)$. A similar technique has already been used for the coagulation equation with multiplicative kernel $K(x,y)=xy$ and without fragmentation \cite{MP04, NZ11, vRS01}. However, the fragmentation term complicates the analysis as it adds a singular reaction term in the first-order hyperbolic equation \eqref{i8} solved by the Laplace transform $L$. This additional difficulty is met again later on in the proof when the method of characteristics is used. Indeed, in contrast to the case without fragmentation, it is no longer a single ordinary differential equation which shows up in the study of characteristics but a nonlinear system of two ordinary differential equations with a singularity. To be more precise, the strategy to show that \eqref{i8} has a global solution $L$ satisfying $L(t,0)=1$ for all times $t\ge 0$ is the following: we employ the method of characteristics to establish that the equation 
$$
\partial_t \tilde{L}(t,s) = (1+k-\tilde{L}(t,s))\ \partial_s \tilde{L}(t,s) + k\  \frac{1-\tilde{L}(t,s)}{s}\ , \quad t>0\ ,\ s>0\ , 
$$
(which is nothing but \eqref{i8} where we have replaced $L(t,0)$ by $1$) has a global solution $\tilde{L}$ satisfying $\tilde{L}(t,0)=1$ for $t\ge 0$. Setting $L=\tilde{L}$ obviously gives a global solution $L$ to \eqref{i8} satisfying $L(t,0)=1$ for $t\ge 0$. As already pointed out, the method of characteristics requires a detailed analysis of a nonlinear system of two ordinary differential equations with a singularity which turns out to be rather involved and is performed in Section~\ref{sec31}. As a consequence of this analysis we obtain the existence of a global solution $L$ to \eqref{i8} satisfying $L(t,0)=1$ for $t\ge 0$ in Section~\ref{sec32}. The connection with \eqref{i7} is then made in Section~\ref{sec4} where we show that the just obtained solution $L$ to \eqref{i8} is completely monotone for all positive times and thus the Laplace transform of a probability measure. Several auxiliary results on completely monotone functions are needed for this step. The existence of a mass-conserving solution to \eqref{i7} in the sense of Definition~\ref{defws} results from the outcome of Sections~\ref{sec3} and~\ref{sec4}. Additional information can be retrieved from the detailed study of $L$ performed in Sections~\ref{sec3} and~\ref{sec4}. This allows us to identify the large time behavior of $L$ in Section~\ref{sec5} as well as the behavior of $\partial_s L(t,s)$ as $s\to 0$ in Section~\ref{sec6}, thereby providing the large time limits \eqref{i12} and \eqref{i13} stated in Theorem~\ref{thint0}.

\section{Alternative representation}\label{sec2}

Let $\nu$ be a weak solution to \eqref{i7} in the sense of Definition~\ref{defws}. Introducing its Laplace transform 
\begin{equation*}
L(t,s) := \int_0^\infty e^{-sx} \nu(t,dx)\ , \quad (t,s)\in [0,\infty)\times (0,\infty)\ , 
\end{equation*}
and observing that $x\mapsto e^{-sx}$ is bounded and continuous for $s>0$, we infer from \eqref{i10} that $L$ solves
\begin{eqnarray}
\partial_t L(t,s) & = & (L(t,0)+k-L(t,s))\ \partial_s L(t,s) + k\  \frac{L(t,0)-L(t,s)}{s}\ , \quad t>0\ ,\ s>0\ , \label{rep3} \\
L(0,s) & = & L_0(s)\ , \quad s>0\ , \label{rep4}
\end{eqnarray}
where
\begin{equation}
L_0(s) := \int_0^\infty e^{-sx} \nu^{in}(dx) \ , \quad s>0\ . \label{rep5}
\end{equation}
Introducing the characteristic equation
\begin{equation}
\frac{dS}{dt}(t) = L(t,S(t)) - L(t,0) - k\ , \label{rep6}
\end{equation}
we infer from \eqref{rep3} that
\begin{equation}
\frac{d}{dt} L(t,S(t)) = k\ \frac{L(t,0)-L(t,S(t))}{S(t)}\ . \label{rep7}
\end{equation}
Consequently, $t\mapsto (S(t),L(t,S(t))$ solves the differential system \eqref{rep6}-\eqref{rep7} which has a singularity when $S(t)$ vanishes and is not closed as it features the yet unknown time dependent function $L(t,0)$. Nevertheless, as long as the total mass is conserved, it follows from \eqref{i7} that $L(t,0)=1$ and the differential system \eqref{rep6}-\eqref{rep7} is closed and can in principle be solved. 

\section{Well-posedness}\label{sec3}

We first list some useful properties of the Laplace transform $L_0$ of $\nu^{in}$ defined in \eqref{rep5}. Owing to \eqref{i7}, $L_0\in C([0,\infty))\cap C^\infty((0,\infty))$ and satisfies
\begin{subequations}
\begin{eqnarray}
& & L_0'(s) = - \int_0^\infty x e^{-sx} \nu^{in}(dx) \;\text{ and }\; 0 < L_0(s) < 1\ , \quad s>0\ , \label{wp0a} \\
& & s L_0'(s) + 1 - L_0(s) = \int_0^\infty \left( e^{sx} - 1- sx \right) e^{-sx} \nu^{in}(dx)  > 0\ , \quad s>0\ , \label{wp0b}
\end{eqnarray}
\end{subequations}
the second statement being a consequence of the elementary inequality $e^{sx}\ge 1+sx$ for $x>0$ and $s>0$. For further use, we also define
\begin{equation}
L_1(s) := \frac{L_0(s)-1}{s}<0 \;\;\text{ with }\;\; L_1'(s) = \frac{s L_0'(s) + 1 - L_0(s)}{s^2}>0\ , \quad s>0\ , \label{wp0c}
\end{equation}
the positivity properties of $-L_1$ and $L_1'$ being a consequence of \eqref{wp0a} and \eqref{wp0b}.

\subsection{An auxiliary differential system}\label{sec31}

According to the previous discussion, we focus in this section on the following initial value problem: given $s>0$,
\begin{eqnarray}
\frac{d\Sigma}{dt}(t) & = & \ell(t) - 1 - k\ , \label{wp1} \\
\frac{d\ell}{dt}(t) & = & k\ \frac{1-\ell(t)}{\Sigma(t)}\ , \label{wp2} \\
(\Sigma,\ell)(0) & = & (s,L_0(s))\ . \label{wp3} 
\end{eqnarray}
We infer from the Cauchy-Lipschitz theorem that there is a unique maximal solution $(\Sigma,\ell)(\cdot,s)\in C([0,T(s));\mathbb{R	}^2)$ to \eqref{wp1}-\eqref{wp3} such that
\begin{equation}
\Sigma(t,s)>0\ , \quad t\in [0,T(s))\ . \label{wp4}
\end{equation}
In addition, 
\begin{equation}
T(s)<\infty \iff \lim_{t\to T(s)} \Sigma(t,s)=0 \;\text{ or }\; \lim_{t\to T(s)} \left( \Sigma(t,s) + |\ell(t,s)| \right) = \infty\ . \label{wp5}
\end{equation}
Since $L_0(s)\in (0,1)$ by \eqref{wp0a}, a first consequence of \eqref{wp2} and the comparison principle is that $\ell(t,s)<1$ for $t\in [0,T(s))$. This fact and \eqref{wp4} ensure that the right hand side of \eqref{wp2} is then positive, hence
\begin{equation}
\partial_t \ell(t,s)>0 \;\;\text{ and }\;\; 0 \le \ell(t,s) < 1\ , \quad t\in [0,T(s))\ . \label{wp6a}
\end{equation}
We then deduce from \eqref{wp1}, \eqref{wp6a}, and the positivity of $k$ that
\begin{equation}
\partial_t \Sigma(t,s) < -k < 0 \;\text{ and }\; \Sigma(t,s)\le s\ , \quad t\in [0,T(s))\ . \label{wp6b}
\end{equation}
Two interesting consequences can be drawn from the estimates \eqref{wp6a} and \eqref{wp6b}: they clearly exclude the occurrence of finite time blowup and imply that $\Sigma(\cdot,s)$ vanishes at a finite time. Recalling \eqref{wp5}, we conclude that 
$$
T(s)< \infty \;\text{ and }\; \lim_{t\to T(s)} \Sigma(t,s)=0\ .
$$

In fact, $\Sigma(\cdot,s)$, $\ell(\cdot,s)$ and $T(s)$ can be computed explicitly as we show now. Owing to \eqref{wp4} and \eqref{wp6a}, the function $\ln{\Sigma} - \ln{(1-\ell)} + (\ell/k)$ is well-defined on $[0,T(s))$ and we infer from \eqref{wp1} and \eqref{wp2} that
$$
\frac{d}{dt} \left( \ln{\Sigma} - \ln{(1-\ell)} + \frac{\ell}{k} \right)(t,s) = 0\ , \quad t\in [0,T(s)\ .
$$
Therefore
\begin{subequations}
\begin{equation}
\Sigma(t,s) = s\ \frac{1-\ell(t,s)}{1-L_0(s)}\ e^{(L_0(s)-\ell(t,s))/k}\ , \quad t\in [0,T(s))\ , \label{wp7a}
\end{equation}
or, alternatively,
\begin{equation}
\frac{1-\ell(t,s)}{\Sigma(t,s)} = \frac{1-L_0(s)}{s} e^{(\ell(t,s)-L_0(s))/k}\ , \quad t\in [0,T(s))\ . \label{wp7b}
\end{equation} \label{wp7}
\end{subequations}
Inserting \eqref{wp7b} in \eqref{wp2} gives 
\begin{equation}
\frac{d\ell}{dt}(t,s)  = k\ \frac{1-L_0(s)}{s}\ e^{-L_0(s)/k} e^{\ell(t,s)/k}\ , \quad t\in [0,T(s))\ ,
\label{wp8b}
\end{equation}
whence, after integration,
\begin{equation}
\ell(t,s) = L_0(s) - k \ln{(1+tL_1(s))}\ , \quad t\in [0,T(s))\ , \label{wp8}
\end{equation}
the function $L_1$ being defined in \eqref{wp0c}. We next combine \eqref{wp7a} and \eqref{wp8} to obtain 
\begin{equation}
\Sigma(t,s) = -\frac{1+t L_1(s)}{L_1(s)}\ \left[ 1 - L_0(s) + k\ \ln{(1+tL_1(s))} \right]\ , \quad t\in [0,T(s))\ . \label{wp9}
\end{equation}
The first term of the right hand side of \eqref{wp9} vanishes when $t=-1/L_1(s)=s/(1-L_0(s))$ while the second term is a decreasing function of time (recall that $L_0(s)<1$) and ranges in $(-\infty,1-L_0(s)]$ when $t$ ranges in $[0,-1/L_1(s))$. It thus vanishes only once in $[0,-1/L_1(s))$ and since we have already excluded finite time blowup, we conclude that $T(s)$ is the unique zero in $[0,-1/L_1(s))$ of the second term of the right hand side of \eqref{wp9}, that is, $T(s)$ solves
\begin{subequations}
\begin{equation}
1 - L_0(s) + k\ \ln{(1+ T(s) L_1(s))} = 0\ , \label{wp10a}
\end{equation}
or, alternatively,
\begin{equation}
T(s) = \frac{s}{1-L_0(s)}\ \left( 1 - e^{(L_0(s)-1)/k} \right) < \frac{s}{1-L_0(s)}\ . \label{wp10b} 
\end{equation}
\end{subequations}
The last bound implies in particular that
\begin{equation}
s(1+t L_1(s)) = s + (L_0(s)-1) t > s e^{(L_0(s)-1)/k} > 0\ , \quad t\in [0,T(s))\ . \label{wp10c}
\end{equation}

Recalling that the differential system \eqref{wp3}-\eqref{wp4} features a singularity as $t\to T(s)$, let us investigate further the behavior of $(\Sigma,\ell)(t,s)$ as $t\to T(s)$.

\begin{lemma}\label{lewp1b}
The functions $\Sigma(\cdot,s)$ and $\ell(\cdot,s)$ both belong to $C^1([0,T(s)])$ with
\begin{align*}
\Sigma(T(s),s) = 0\ , & \quad \ell(T(s),s) = 1\ , \\ 
\partial_t \Sigma(T(s),s) = - k\ , & \quad \partial_t \ell(T(s),s) = - k L_1(s) e^{(L_0(s)-1)/k}\ . 
\end{align*}
\end{lemma}

\begin{proof}
The behavior of $\Sigma(t,s)$ and $\ell(t,s)$ as $t\to T(s)$ is a straightforward consequence of \eqref{wp8}, \eqref{wp9}, and \eqref{wp10a}. Next, it readily follows from \eqref{wp1} that
$$
\lim_{t\to T(s)} \partial_t \Sigma(t,s) = -k \ .
$$
Consequently, $\Sigma(\cdot,s)\in C^1([0,T(s)])$ and $\ell(\cdot, s)$ shares the same regularity thanks to a similar argument relying on \eqref{wp8b}.
\end{proof}

We now study the behavior of $T(s)$ as a function of $s>0$ and gather the outcome of the analysis in the next lemma.

\begin{lemma}\label{lewp2}
The function $T$ is an increasing $C^\infty$-smooth diffeomorphism from $(0,\infty)$ onto $(0,\infty)$ and enjoys the following properties:
\begin{subequations}
\begin{eqnarray}
& & \lim_{s\to 0} T(s) = 0 \;\text{ and }\; T(s) \sim \frac{s}{k} \;\text{ as } s \to 0\ , \label{wp11a} \\
& & \lim_{s\to \infty} T(s) = \infty \;\text{ and }\; T(s) \sim (1- e^{-1/k}) s \;\text{ as } s \to \infty\ . \label{wp11b}
\end{eqnarray}
\end{subequations}
\end{lemma}

\begin{proof}
The smoothness of $T$ readily follows from that of $L_0$ and \eqref{wp0a} by \eqref{wp10b} while \eqref{wp11a} and \eqref{wp11b} are consequences of \eqref{wp10b} and the properties of $L_0$. To establish the monotonicity of $T$, we  differentiate \eqref{wp10a} with respect to $s$ and find 
$$
-L_0'(s) + k\ \frac{1 + L_1'(s) T(s) + L_1(s) T'(s)}{1 + L_1(s) T(s)} = 0\ .
$$
Since $1+L_1(s) T(s) = e^{(L_0(s)-1)/k}$ by \eqref{wp10a}, we obtain
$$
- L_0'(s) e^{(L_0(s)-1)/k} + k \left( 1+ L_1'(s) T(s) + L_1(s) T'(s) \right) = 0\ ,
$$
and then
$$
- k L_1(s) T'(s) = - L_0'(s) e^{(L_0(s)-1)/k} + k ( 1+ L_1'(s) T(s))\ .
$$
Since $-L_0'$ and $L_1'$ are both positive by \eqref{wp0a} and \eqref{wp0c}, we conclude that
\begin{equation*}
T'(s)>0\ , \quad s>0\ , 
\end{equation*}
which, together with \eqref{wp11a} and \eqref{wp11b}, implies that $T$ is an increasing $C^\infty$-smooth diffeomorphism from $(0,\infty)$ onto $(0,\infty)$ and completes the proof.
\end{proof}

Denoting the inverse of $T$ by $T^{-1}$, we deduce from \eqref{wp10a} that
\begin{equation}
1 = L_0(T^{-1}(t)) - k \ln{(1+t L_1(T^{-1}(t))}\ , \quad t>0\ . \label{wp10d}
\end{equation}

We next turn to the properties of $\Sigma$ with respect to the variable $s$ and establish the following monotonicity result:

\begin{lemma}\label{lewp3}
Let $t>0$. The function $\Sigma(t,\cdot)$ is an increasing $C^\infty$-smooth diffeomorphism from $(T^{-1}(t),\infty)$ onto $(0,\infty)$ and satisfies 
\begin{equation}
\partial_s \Sigma (t,s)> 0 \;\;\text{ for }\;\; s\in (T^{-1}(t),\infty) \;\;\text{ and }\;\; \Sigma(t,s)\sim s \;\;\text{ as }\;\; s\to \infty\ . \label{wp14}
\end{equation}
\end{lemma}

\begin{proof}
Let $s>0$ and $t\in (0,T(s))$. We differentiate \eqref{wp9} with respect to $s$ and obtain
\begin{align*}
\partial_s \Sigma(t,s) = & - \frac{1+t L_1(s)}{L_1(s)} \left[ -L_0'(s) + kt\ \frac{L_1'(s)}{1+t L_1(s)} \right] \\
+ & \frac{L_1'(s)}{L_1(s)^2}\ \left( 1 - L_0(s) + k \ln{ (1+t L_1(s))} \right) \\
= & \frac{L_0'(s)}{L_1(s)} + t L_0'(s) - kt \frac{L_1'(s)}{L_1(s)} - \frac{s L_1'(s)}{L_1(s)} + k\ \frac{L_1'(s)}{L_1(s)^2} \ln{(1+t L_1(s))} \ ,
\end{align*}
\begin{equation}
\partial_s \Sigma(t,s) = 1 + t L_0'(s) + k\ \frac{L_1'(s)}{L_1(s)^2} \ \left[ \ln{(1+t L_1(s))} - t L_1(s) \right]\ . \label{wp13}
\end{equation}
Differentiating \eqref{wp13} with respect to $t$ gives
\begin{align*}
\partial_t \partial_s \Sigma(t,s) = & L_0'(s) + k\ \frac{L_1'(s)}{L_1(s)^2}\ \left[ \frac{L_1(s)}{1 + t L_1(s)} - L_1(s) \right] \\
= & L_0'(s) - kt\ \frac{L_1'(s)}{1 + t L_1(s))}< 0
\end{align*}
for $t\in [0,T(s))$, the negativity of the right hand side of the above inequality being a consequence of \eqref{wp0a}, \eqref{wp0c}, and \eqref{wp10c}. Therefore, for all $t\in [0,T(s))$, 
\begin{equation}
\partial_s \Sigma(t,s) > \tau(s) := \lim_{t\to T(s)} \partial_s \Sigma(t,s)\ , \label{wp13b}
\end{equation} 
where 
$$
\tau(s) = 1 + T(s) L_0'(s) + k\ \frac{L_1'(s)}{L_1(s)^2} \ \left[ \ln{(1+T(s) L_1(s))} - L_1(s) T(s) \right]\ .
$$
Since 
$$
\ln{(1+T(s) L_1(s))} = \frac{L_0(s)-1}{k} = \frac{s L_1(s)}{k}
$$
by \eqref{wp10a}, we realize that
\begin{align*}
\tau(s) = & 1 + T(s) L_0'(s) + k\ \frac{L_1'(s)}{L_1(s)} \ \left[ \frac{s}{k} - T(s) \right] \\
= & L_0'(s)\ \left( T(s) - \frac{s}{1-L_0(s)} \right) - k\ \frac{L_1'(s)}{L_1(s)} \ T(s) > 0\ ,
\end{align*}
due to \eqref{wp0a}, \eqref{wp0c}, and \eqref{wp10b}. Recalling \eqref{wp13b}, we have shown that $\partial_s \Sigma(t,s)>0$ for $t\in (0,T(s))$ and $s>0$ or equivalently for $s\in (T^{-1}(t),\infty)$ and $t>0$. Consequently, for each $t>0$, $\Sigma(t,\cdot)$ is an increasing $C^\infty$-smooth diffeomorphism from $(T^{-1}(t),\infty)$ onto its range which is nothing but $(0,\infty)$ since $\Sigma(t,s)\sim s$ as $s\to\infty$ by \eqref{wp9} and $\Sigma(t,s)\to 0$ as $s\to T^{-1}(t)$ by Lemma~\ref{lewp1b}. The proof of Lemma~\ref{lewp3} is then complete.
\end{proof}

For $t>0$, we denote the inverse of $\Sigma(t,\cdot)$ by $\zeta(t,\cdot)$ and observe that it is a $C^\infty$-smooth increasing function from $(0,\infty)$ onto $(T^{-1}(t),\infty)$. Since 
\begin{equation}
\Sigma(t,\zeta(t,s))=s \;\text{ for }\; (t,s)\in (0,\infty)^2\ , \label{wp100}
\end{equation} 
we infer from \eqref{wp14}, \eqref{wp100}, and the implicit function theorem that $\zeta\in C^\infty((0,\infty)^2)$ with
\begin{equation}
\partial_s \zeta(t,s) = \frac{1}{\partial_s \Sigma(t,\zeta(t,s))} \;\text{ and }\; \partial_t \zeta(t,s) = - \frac{\partial_t \Sigma(t,\zeta(t,s))}{\partial_s \Sigma(t,\zeta(t,s))}\ , \quad (t,s)\in (0,\infty)^2\ . \label{wp101}
\end{equation}

We end up this section with the differentiability of $\Sigma(t,\cdot)$ and $\ell(t,\cdot)$ at $s=T^{-1}(t)$.

\begin{lemma}\label{lewp4}
Let $t>0$. Both $\Sigma(t,\cdot)$ and $\ell(t,\cdot)$ belong to $C^1\left( [T^{-1}(t),\infty) \right)$.
\end{lemma}

\begin{proof}
Since $T^{-1}(t)>0$ by Lemma~\ref{lewp2} and $L_0, L_1\in C^1((0,\infty))$, we infer from \eqref{wp13} that $\partial_s \Sigma(t,s)$ has a limit as $s\to T^{-1}(t)$ and thus that $\Sigma(t,\cdot)\in C^1([T^{-1}(t),\infty)$. Similarly, by \eqref{wp8}, 
$$
\partial_s \ell(t,s) = L_0'(s) - kt\ \frac{L_1'(s)}{1+t L_1(s)}\ ,
$$
which has a limit as $s\to T^{-1}(t)$ as $L_0$ and $L_1$ belong to $C^1((0,\infty))$.
\end{proof}

\subsection{Existence of a solution to \eqref{rep3}-\eqref{rep4}}\label{sec32}

After this preparation, we are in a position to show the existence of a solution to \eqref{rep3}-\eqref{rep4}. As expected from the analysis performed in Section~\ref{sec2}, we set 
\begin{equation}
L(t,s) := \ell(t,\zeta(t,s))\ , \quad (t,s)\in (0,\infty)^2\ , \label{wp102}
\end{equation}
and aim at showing that $L$ solves \eqref{rep3}-\eqref{rep4} as well as identifying $L(t,0)$ for all $t>0$. 

On the one hand, the properties of $\zeta$, \eqref{wp8}, \eqref{wp10d}, and \eqref{wp102} entail that, for $t>0$,
\begin{align*}
L(t,0) = &\ \lim_{s\to 0} L(t,s) = \lim_{s\to 0} \ell(t,\zeta(t,s)) = \lim_{s\to T^{-1}(t)} \ell(t,s) \\
= &\ L_0(T^{-1}(t)) - k \ln{(1+t L_1(T^{-1}(t)))} =1\ ,
\end{align*}
\begin{equation}
L(t,0) = 1\ , \qquad t>0\ . \label{wp103}
\end{equation}

On the other hand, it follows from \eqref{wp1}, \eqref{wp2}, \eqref{wp101}, and \eqref{wp102} that, for $(t,s)\in (0,\infty)^2$,
\begin{align*}
\partial_t L(t,s) = & k\ \frac{1 - L(t,s)}{s} + \partial_s \ell(t,\zeta(t,s))\ \partial_t\zeta(t,s) \\
= & k\ \frac{1 - L(t,s)}{s} - \partial_s \ell(t,\zeta(t,s))\ \frac{\partial_t\Sigma(t,\zeta(t,s))}{\partial_s\Sigma(t,\zeta(t,s))} \\
= & k\ \frac{1 - L(t,s)}{s} - \partial_s L(t,s)\ (L(t,s)-k-1)\ ,
\end{align*}
whence, thanks to \eqref{wp103},
\begin{equation*}
\partial_t L(t,s) = \partial_s L(t,s)\ (L(t,0)+k-L(t,s)) + k\ \frac{L(t,0) - L(t,s)}{s}\ , \quad (t,s)\in (0,\infty)^2\ . 
\end{equation*}
Finally, the properties of $\Sigma$, $\zeta$, and \eqref{wp3} imply that, for $s>0$,
$$
\lim_{t\to 0} L(t,s) = \lim_{t\to 0} \ell(t,\zeta(t,s)) = \ell(0,s) = L_0(s)\ .
$$
We have thus shown that the function $L$ defined in \eqref{wp102} solves \eqref{rep3}-\eqref{rep4} and enjoys the additional property \eqref{wp103}. To obtain a solution to the coagulation-fragmentation equation \eqref{i6} it remains to show that, for all $t>0$, $L(t,\cdot)$ is the Laplace transform of a non-negative bounded measure or alternatively that it is completely monotone. This will be the aim of Section~\ref{sec4}.

\section{Complete monotonicity}\label{sec4}

\begin{proposition}\label{prcm0}
For all $t>0$, the function $L(t,\cdot)$ defined in \eqref{wp102} is completely monotone.
\end{proposition}

We first recall two important criteria guaranteeing complete monotonicity, see \cite[Chapter~XIII.4, Criterion~1 \& Criterion~2]{Fe71} for instance. 

\begin{lemma}\label{lecm1}
\begin{enumerate}
\item If $\varphi$ and $\psi$ are completely monotone functions then their product $\varphi \psi$ is also a completely monotone function.
\item If $\varphi$ is a completely monotone function and $\psi$ is a non-negative function with a completely monotone derivative, then $\varphi \circ \psi$ is a completely monotone function.
\end{enumerate}
\end{lemma}

In particular, the complete monotonicity of $r\mapsto -\ln{r}$ and Lemma~\ref{lecm1}~(2) have the following consequence.

\begin{lemma}\label{lecm2}
Let $I$ be an open interval of $\RR$ and $g: I \mapsto (0,1)$ be a $C^\infty$-smooth function such that $g'$ is completely monotone. Then $-\ln{g}$ is completely monotone. 
\end{lemma}

Let $t>0$. Recalling that $L(t,\cdot)$ is given by
$$
L(t,s) = \ell(t,\zeta(t,s))\ , \qquad s>0\ ,
$$
the proof of its complete monotonicity is performed in two steps. More precisely, we prove that $\ell(t,\cdot)$ and the derivative of $\zeta(t,\cdot)$ are completely monotone. The complete monotonicity of $L(t,\cdot)$ is then a consequence of Lemma~\ref{lecm1}~(2) and the non-negativity of $\zeta(t,\cdot)$. To prove the complete monotonicity of $\ell(t,\cdot)$ we need the following result:

\begin{lemma}\label{lecm3}
Let $g: (0,\infty) \mapsto (0,1)$ be a completely monotone function satisfying $g(0)=1$. Setting $G(x) := (g(x)-1)/x$ for $x\in (0,\infty)$, its derivative $G'$ is completely monotone.
\end{lemma}

\begin{proof}
We first note that
\begin{equation}
G'(x) = \frac{x g'(x) - g(x) + 1}{x^2} = \frac{g'(x) - G(x)}{x}\,, \qquad x\in (0,\infty)\,. \label{cm1}
\end{equation}

\noindent\textbf{Step~1:} We first prove by induction that, for $n\ge 0$,  
\begin{equation}
G^{(n+1)}(x) = \frac{g^{(n+1)}(x) - (n+1) G^{(n)}(x)}{x}\,, \qquad x\in (0,\infty)\,, \label{cm2}
\end{equation}
with the obvious notation $G^{(0)}=G$. Indeed, \eqref{cm2} is clearly true for $n=0$ by \eqref{cm1}. Assume next that \eqref{cm2} is true for some $n\ge 0$. Differentiating the corresponding identity gives, for $x>0$,
\begin{equation*}
G^{(n+2)}(x) = \frac{g^{(n+2)}(x) - (n+1) G^{(n+1)}(x)}{x} -  \frac{g^{(n+1)}(x) - (n+1) G^{(n)}(x)}{x^2}
\end{equation*}
We next use \eqref{cm2} for $n$ and find
\begin{equation*}
G^{(n+2)}(x) = \frac{g^{(n+2)}(x) - (n+1) G^{(n+1)}(x)}{x} -  \frac{G^{(n+1)}(x)}{x}\,,
\end{equation*}
whence \eqref{cm2} for $n+1$.

\noindent\textbf{Step~2:} We next prove by induction that, for $n\ge 1$, 
\begin{equation}
G^{(n)}(x) = \frac{n!}{x^{n+1}}\ \left[ (-1)^n\ (g(x)-1) + \sum_{j=1}^n (-1)^{n-j}\ \frac{x^j}{j!}\ g^{(j)}(x) \right]\,, \quad x\in (0,\infty)\,. \label{cm3}
\end{equation}
Indeed, the identity \eqref{cm3} is true for $n=1$ by \eqref{cm1}. Assume next that $G^{(n)}$ is given by \eqref{cm3} for some $n\ge 1$. It then follows from \eqref{cm2} and \eqref{cm3} that, for $x>0$,
\begin{align*}
G^{(n+1)}(x) = & \frac{g^{(n+1)}(x)}{x} - \frac{(n+1)!}{x^{n+2}}\ \left[ (-1)^n\ (g(x)-1) + \sum_{j=1}^n (-1)^{n-j}\ \frac{x^j}{j!}\ g^{(j)}(x) \right] \\
= & \frac{(n+1)!}{x^{n+2}}\ \left[ (-1)^{n+1}\ (g(x)-1) + \sum_{j=1}^n (-1)^{n+1-j}\ \frac{x^j}{j!}\ g^{(j)}(x) + \frac{x^{n+1}}{(n+1)!}\ g^{(n+1)}(x) \right]\,.
\end{align*}
Thus, $G^{(n+1)}$ is also given by \eqref{cm3}.

\noindent\textbf{Step~3:} Define $h_n(x) := x^{n+1}\ G^{(n)}(x) / n!$ for $n\ge 0$ and $x>0$. Thanks to \eqref{cm3}, 
$$
h_n(x) = (-1)^n\ (g(x)-1) + \sum_{j=1}^n (-1)^{n-j}\ \frac{x^j}{j!}\ g^{(j)}(x)\,, \quad x\in (0,\infty)\,. 
$$
Since $g(0)=1$, we have $h_n(0)=0$ and, for $x>0$, 
\begin{align*}
h_n'(x) = & (-1)^n\ g'(x) + \sum_{j=1}^n (-1)^{n-j}\ \frac{x^j}{j!}\ g^{(j+1)}(x) + \sum_{j=1}^n (-1)^{n-j}\ \frac{x^{j-1}}{(j-1)!}\ g^{(j)}(x) \\
= & \sum_{j=1}^{n+1} (-1)^{n+1-j}\ \frac{x^{j-1}}{(j-1)!}\ g^{(j)}(x) + \sum_{j=1}^n (-1)^{n-j}\ \frac{x^{j-1}}{(j-1)!}\ g^{(j)}(x) \\
= & \frac{x^n}{n!}\ g^{(n+1)}(x)\,.
\end{align*}
Since $g$ is completely monotone, the previous identity implies that $(-1)^{n+1}\ h_n' \ge 0$. Recalling that $h_n(0)=0$, we have thus shown that $(-1)^{n+1}\ h_n\ge 0$ which in turn gives $(-1)^{n+1}\ G^{(n)}\ge 0$ and the complete monotonicity of $G'$.
\end{proof}

\begin{lemma}\label{lecm4}
For each $t>0$, $\ell(t,\cdot)$ is completely monotone in $(T^{-1}(t),\infty)$. 
\end{lemma}

\begin{proof}
Fix $t>0$ and recall that $\ell(t,\cdot)$ is given by
$$
\ell(t,s) = L_0(s) - k \ln{( 1 + t L_1(s))}\ , \quad s\in (T^{-1}(t),1)\ ,
$$
with $L_1(s) := (L_0(s)-1)/s$ for $s>0$. Since $L_0$ is completely monotone with $L_0(0)=1$, we infer from Lemma~\ref{lecm3} that the derivative of the function $L_1$ is completely monotone in $(0,\infty)$. Then, so is the derivative of $1+tL_1$ and $s\mapsto 1+tL_1(s)$ ranges in $(0,1)$ when $s\in (T^{-1}(t),\infty)$ according to \eqref{wp10c}. We are thus in a position to apply Lemma~\ref{lecm2} and conclude that $s\mapsto - \ln{( 1 + tL_1(s))}$ is completely monotone in $(T^{-1}(t),\infty)$. Since $k>0$ and $L_0$ is completely monotone in $(0,\infty)$, we conclude that $\ell(t,\cdot)$ is completely monotone in $(T^{-1}(t),\infty)$.
\end{proof}

We next turn to $\zeta(t,\cdot)$ and first establish the following auxiliary result.

\begin{lemma}\label{lecm5}
Let $g$ be a non-negative function in $C^\infty(0,\infty)$ such that $g'$ is completely monotone and $g'<1$. Then the function $(\id-g)^{-1}$ has a completely monotone derivative.
\end{lemma}

\begin{proof}
For the sake of completeness, we give a sketch of the proof which is actually outlined in \cite[p.~1209]{MP04}. We set $\sigma_\infty := (\id - g)^{-1}$ and define by induction a sequence of functions $(\sigma_n)_{n\ge 0}$ as follows: 
\begin{equation}
\sigma_0(s) := s \;\;\text{ and }\;\; \sigma_{n+1}(s) := s + g(\sigma_n(s))\,, \qquad s>0\,, \quad n\ge 0\,. \label{cm40}
\end{equation}
On the one hand, thanks to the properties of $g$ (and in particular the bounds $0\le g'<1$), the function $\sigma_n$ is well-defined and non-negative for all $n\ge 0$ and belongs to $C^\infty(0,\infty)$. In addition, it satisfies 
\begin{equation}
\sigma_0(s) = s \le \sigma_n(s) \le \sigma_{n+1}(s) \le \sigma_\infty(s)\,, \qquad s>0\,, \quad n\ge 0\,. \label{cm4}
\end{equation}
Now, fix $s_1>0$ and $s_2>s_1$. Owing to the property $g'<1$, there is $\delta\in (0,1)$ such that $g'(s)\le \delta$ for each $s\in [s_1,\sigma_\infty(s_2)]$. Since $\sigma_n(s)\in [s_1,\sigma_\infty(s_2)]$ for $n\ge 0$ and $s\in [s_1,s_2]$ by \eqref{cm4}, we obtain
$$
0 \le \sigma_\infty(s) - \sigma_{n+1}(s) = g(\sigma_\infty(s)) - g(\sigma_{n}(s)) \le \delta\ \left( \sigma_\infty(s) - \sigma_{n}(s)\right)\,.
$$
This estimate readily implies that 
\begin{equation}
(\sigma_n)_{n\ge 0} \;\text{ converges uniformly towards }\; \sigma_\infty \;\text{ on compact subsets of }\; (0,\infty)\,.
\label{cm5}
\end{equation}

On the other hand, it follows from \eqref{cm40} by induction that \begin{equation}
\sigma_n \;\text{ is completely monotone for every }\; n\ge 0\,.
\label{cm6}
\end{equation} 
Indeed, $\sigma_0=\id$ is clearly completely monotone and, if $\sigma_n$ is assumed to be completely monotone, the complete monotonicity of $\sigma_{n+1}$ follows from \eqref{cm40} and that of $\sigma_n$ and $g'$ with the help of \cite[Chapter~XIII.4, Criterion~2]{Fe71}.

The assertion of Lemma~\ref{lecm5} then follows from \eqref{cm5} and \eqref{cm6} by \cite[Chapter~XIII.1, Theorem~2 \& Chapter~XIII.4, Theorem~1]{Fe71}.
\end{proof}

\begin{lemma}\label{lecm6}
For each $t>0$, $\partial_s \zeta(t,\cdot)$ is completely monotone in $(0,\infty)$.
\end{lemma}

\begin{proof}
Fix $t>0$. Introducing
$$
\Phi(s) := s - \Sigma(t,s) = (1-L_0(s)) t + k\ \frac{1+t L_1(s)}{L_1(s)} \ln{(1+tL_1(s))}\ , \quad s\in (T^{-1}(t),\infty)\ ,
$$
the formula \eqref{wp9} also reads $\Sigma(t,s) = s -\Phi(s)$ for $s\in (T^{-1}(t),\infty)$ from which we deduce that
\begin{equation}
s = \zeta(t,s) - \Phi(\zeta(t,s))\ , \quad s>0\ . \label{wp16}
\end{equation}
Therefore, $\zeta(t,\cdot)$ satisfies a functional identity of the form required to apply Lemma~\ref{lecm5}. To go on, we have to show that $\Phi'$ is completely monotone. To this end, we compute $\Phi'$ and find
$$
\Phi'(s) = - t L_0'(s) + k\ \frac{L_1'(s)}{L_1(s)^2}\ \left[ t L_1(s) - \ln{(1+t L_1(s))} \right]\ , \quad \quad s\in (T^{-1}(t),\infty)\ .
$$
Given $\theta\in [0,t]$, the monotonicity of $T^{-1}$ ensures that $T^{-1}(\theta)\le T^{-1}(t)$ so that $1+\theta L_1>0$ in $(T^{-1}(t),\infty)$ by \eqref{wp10c}. In addition, $1+\theta L_1$ has a completely monotone derivative since $L_1'$ is completely monotone by Lemma~\ref{lecm3} which, together with the complete monotonicity of $z\mapsto 1/z$ and Lemma~\ref{cm1}~(2) entails the complete monotonicity of $1/(1+\theta L_1)$ in $(T^{-1}(t),\infty)$ for all $\theta\in [0,t]$. Observing that
$$
\frac{t L_1(s) - \ln{(1+t L_1(s))}}{L_1(s)^2} = \int_0^t \frac{\theta}{1+\theta L_1(s)}\ d\theta\ , \quad s\in (T^{-1}(t),\infty)\ ,
$$ 
we infer from \cite[Theorem~4]{MS01} that 
$$
\frac{t L_1 - \ln{(1+t L_1)}}{L_1^2} \;\text{ is completely monotone in }\; (T^{-1}(t),\infty)\ .
$$
Using now Lemma~\ref{cm1}~(1) along with the positivity of $k$ and the complete monotonicity of $L_1'$ and $-L_0'$ entails that $\Phi'$ is completely monotone in $(T^{-1}(t),\infty)$. Furthermore, recalling the definition of $\Phi$, there holds $\Phi'=1 - \partial_s \Sigma(t,\cdot)<1$ by \eqref{wp14}.

Summarizing, we have shown that $\Phi\in C^\infty((T^{-1}(t),\infty))$ has a completely monotone derivative $\Phi'$ satisfying $\Phi'<1$ while $\zeta(t,\cdot)$ solves \eqref{wp16}. Lemma~\ref{lecm5} then ensures that $\partial_s \zeta(t,\cdot)$ is completely monotone in $(0,\infty)$.
\end{proof}

\begin{proof}[Proof of Proposition~\ref{prcm0}]
Fix $t>0$. The complete monotonicity of $L(t,\cdot)$ is a straightforward consequence of the non-negativity of $\zeta(t,\cdot)$, Lemma~\ref{cm1}~(2), Lemma~\ref{lecm4}, and Lemma~\ref{lecm6}.
\end{proof}

\medskip

Now, for each $t>0$, $L(t,\cdot)$ is completely monotone in $(0,\infty)$ with $L(t,0)=1$ by Proposition~\ref{prcm0} and \eqref{wp103} so that it is the Laplace transform of a probability measure $\nu(t)\in \mathfrak{M}^+$ by \cite[Chapter~XIII.4, Theorem~1]{Fe71}. In addition, it follows from the time continuity of $L$ and \cite[Chapter~XIII.1, Theorem~2]{Fe71} that the map $\nu: [0,\infty)\to \mathfrak{M}^+$ is weakly continuous. We finally argue as in \cite[Section~2]{MP04} to show that $\nu$ satisfies \eqref{i10} for all $C^1$-smooth functions $\vartheta$ with compact support. An additional approximation argument allows us to extend the validity of \eqref{i10} to all continuous functions $\vartheta$ with compact support and complete the proof of the first statement of Theorem~\ref{thint0}.  

\section{Large time behavior}\label{sec5}

We now aim at investigating the behavior of $L(t,ts)$ as $t\to \infty$ for any given $s>0$. More specifically, we prove the following convergence result.

\begin{proposition}\label{prltb0}
For all $s>0$, there holds
\begin{equation}
\lim_{t\to \infty} L(t,ts) = 1 + s - k W\left( \frac{s}{k} e^{(1+s)/k} \right)\ , \label{ltb0}
\end{equation}
where we recall that $W$ is the Lambert $W$-function, that is, the inverse function of $z\mapsto z e^z$ in $(0,\infty)$.
\end{proposition}

\begin{proof}
We fix $s>0$ and set
\begin{equation}
\eta(t) := \zeta(t,ts) \;\;\text{ and }\;\; \mu(t) := - t L_1(\eta(t))\ , \quad t>0\ . \label{ltb1}
\end{equation}
Since $\eta(t) = \zeta(t,ts) \ge T^{-1}(t)$, Lemma~\ref{lewp2} ensures that
\begin{equation}
\lim_{t\to\infty} \eta(t) = \infty\ . \label{ltb2}
\end{equation}
Next, for $\sigma>T^{-1}(t)$, we infer from the properties of $L_0$ and \eqref{wp10b} that
$$
1 > 1 + t L_1(\sigma) > 1 + L_1(\sigma) T(\sigma) = e^{(L_0(\sigma)-1)/k}\ .
$$
Taking $\sigma = \eta(t)$ in the above estimate gives
\begin{equation}
1 > 1 - \mu(t) > e^{(L_0(\eta(t))-1)/k} > e^{-1/k}\ , \quad t>0\ . \label{ltb3}
\end{equation}
Now, we infer from \eqref{wp9} (with $\eta(t)=\zeta(t,s)$ instead of $s$) that 
\begin{align*}
ts = & \frac{\eta(t) (1+t L_1(\eta(t)))}{1-L_0(\eta(t))} \left[ 1 - L_0(\eta(t)) + k \ln{(1-\mu(t))} \right] \\
s = & \frac{1-\mu(t)}{\mu(t)} \left[ 1 - L_0(\eta(t)) + k \ln{(1-\mu(t))} \right] \\
\mu(t) s = & (1-\mu(t))\ (1 - L_0(\eta(t))) + k (1-\mu(t)) \ln{(1-\mu(t))}\ ,
\end{align*}
whence
\begin{equation}
s + (1-\mu(t))\ (1 - L_0(\eta(t))) = h(1-\mu(t))\ , \quad t>0\ , \label{ltb4}
\end{equation}
where
\begin{equation}
h(z) := (1+s) z + k z \ln{z}\ , \quad z\in (e^{-1/k},1)\ . \label{ltb5}
\end{equation}
On the one hand, it follows from \eqref{ltb2}, \eqref{ltb3}, and \eqref{ltb4} that
\begin{equation}
\lim_{t\to\infty} h(1-\mu(t)) = s\ . \label{ltb6}
\end{equation}
On the other hand, $h$ is increasing on $(e^{-1/k},1)$ with $h(e^{-1/k}) = s e^{-1/k} < s < s+1 = h(1)$, so that $h$ is a one-to-one function from $(e^{-1/k},1)$ onto $(s e^{-1/k},s+1)$. Introducing its inverse $h^{-1}$, we deduce from \eqref{ltb6}  that
\begin{equation}
\lim_{t\to\infty} \mu(t) = 1 - h^{-1}(s)\ . \label{ltb7}
\end{equation}
Recalling \eqref{ltb1} and \eqref{ltb2}, it follows from \eqref{ltb7} that
\begin{equation}
\lim_{t\to\infty} \frac{\zeta(t,ts)}{t} = \frac{1}{1 - h^{-1}(s)}\ . \label{ltb8}
\end{equation}
Finally, since 
$$
L(t,ts) = \ell(t,\zeta(t,ts)) = L_0(\zeta(t,ts)) - k \ln{(1-\mu(t))}
$$
by \eqref{wp8}, \eqref{wp102}, and \eqref{ltb1}, we infer from \eqref{ltb7}, \eqref{ltb8}, and the properties of $L_0$ that 
\begin{equation}
\lim_{t\to\infty} L(t,ts) = -k \ln{h^{-1}(s)}\ . \label{ltb9}
\end{equation}

We are left with expressing the right hand side of \eqref{ltb9} with the help of the Lambert $W$-function. To this end, we note that $h^{-1}(s)$ solves
$$
h^{-1}(s)\ \left[ (1+s) + k \ln{h^{-1}(s)} \right] = s
$$
by \eqref{ltb5} or, equivalently,
$$
W^{-1}\left( \ln{\left( e^{(1+s)/k} h^{-1}(s) \right)} \right) = \frac{s}{k}\ e^{(1+s)/k}\ .
$$
Therefore,
$$
e^{(1+s)/k} h^{-1}(s) = \exp{\left\{ W\left( \frac{s}{k}\ e^{(1+s)/k} \right) \right\}}\ , 
$$
and applying $-k\ln{}$ to both sides of the above identity leads us to \eqref{ltb0} thanks to \eqref{ltb9}.
\end{proof}

\medskip

We finally use \cite[Chapter~XIII.1, Theorem~2]{Fe71} to express the outcome of Proposition~\ref{prltb0} in terms of $\nu$ and obtain the last statement of Theorem~\ref{thint0}. 

\section{Second moment estimate}\label{sec6}

We establish in this section an interesting smoothing property of \eqref{i6}, namely that the second moment of the solution becomes instantaneously finite for positive times, even it is initially infinite. We also investigate its large time behavior and answer by the positive a conjecture of Vigil \& Ziff \cite{VZ89}. 

\begin{proposition}\label{prsme0}
For $t>0$, there holds
\begin{equation}
\partial_s L(t,0) = \frac{L_1(T^{-1}(t))}{1+ t L_1(T^{-1}(t))} \;\;\text{ and }\;\; \partial_s L(t,0) \mathop{\sim}_{t\to\infty} \frac{1-e^{1/k}}{t}\ . \label{sme0a}
\end{equation}
\end{proposition}

Note that the positivity of $T^{-1}(t)$ for $t>0$ guaranteed by Lemma~\ref{lewp2} ensures that $L_1(T^{-1}(t))$ is finite whatever the value $t>0$. Recalling that $L_1(s)=(L_0(s)-1)/s$, we realize that $L_1(T^{-1}(t))$ however blows down as $t\to 0$ if $L_0'(0)=-\infty$, that is, $f^{in}$ has an infinite second moment.

\begin{proof}
Fix $t>0$ and set $\theta:=T^{-1}(t)$. Recalling that $L(t,s)=\ell(t,\zeta(t,s))$ by \eqref{wp102}, it follows from \eqref{wp101} that  
$$
\partial_s L(t,s) = \frac{\partial_s \ell(t,\zeta(t,s))}{\partial_s \Sigma(t,\zeta(t,s))}\ , \quad s>0\ ,
$$
while \eqref{wp8} and \eqref{wp13} give
\begin{align*}
\partial_s \ell(t,\sigma) = & L_0'(\sigma) - \frac{kt L_1'(\sigma)}{1+t L_1(\sigma)}\ , \quad \sigma>\theta\ , \\
\partial_s \Sigma(t,\sigma) = & 1 + t L_0'(\sigma) + k\ \frac{L_1'(\sigma)}{L_1(\sigma)^2} \ \left[ \ln{(1+t L_1(\sigma))} - t L_1(\sigma) \right]\ , \quad \sigma>\theta\ .
\end{align*}
Owing to Lemma~\ref{lewp4}, we may take $\sigma=\theta$ in the previous identities and use \eqref{wp10d} to obtain
$$
\partial_s \ell(t,\theta) = L_0'(\theta) - \frac{kt L_1'(\theta)}{1+t L_1(\theta)} = \frac{(1+t L_1(\theta)) L_0'(\theta) - kt L_1'(\theta)}{1+t L_1(\theta)}\ ,
$$
and
\begin{align*}
\partial_s \Sigma(t,\theta) = & 1 + t L_0'(\theta) +  \frac{L_1'(\theta)}{L_1(\theta)^2} \ \left( \theta L_1(\theta)) - kt L_1(\theta) \right) \\
= & 1 + t L_0'(\theta) +\frac{\theta L_0'(\theta) + 1 - L_0(\theta)}{L_0(\theta)-1}  - kt \frac{L_1'(\theta)}{L_1(\theta)} \\
= & \left( t + \frac{1}{L_1(\theta)} \right) L_0'(\theta) - kt \frac{L_1'(\theta)}{L_1(\theta)} \\
= & \frac{(1+t L_1(\theta)) L_0'(\theta) - kt L_1'(\theta)}{L_1(\theta)}  \ .
\end{align*}
Combining the above  formulas for $\partial_s L(t,s)$, $\partial_s \ell(t,\theta)$, and $\partial_s \Sigma(t,\theta)$, and recalling that $\zeta(t,0)=\theta$ give the first statement in \eqref{sme0a} and imply in particular that $\partial_s L(t,0)$ is finite. 

To prove the second statement in \eqref{sme0a}, we use the first one to obtain
$$
\frac{1}{|\partial_s L(t,0)|} = - t - \frac{1}{L_1(T^{-1}(t))} = T^{-1}(t) -t + \frac{T(t)^{-1} L_0(T^{-1}(t))}{1 - L_0(T^{-1}(t))}\ .
$$
Since $T^{-1}(t) \sim t/(1-e^{-1/k})$ as $t\to \infty$ by \eqref{wp11b}, it follows from the properties of $L_0$ and the above identity that
$$
\frac{1}{|\partial_s L(t,0)|} \sim \left( \frac{1}{1-e^{-1/k}} - 1 \right) t \;\;\text{ as }\;\; t\to \infty\ ,
$$
and the proof of \eqref{sme0a} is complete.
\end{proof}

\medskip

Since 
$$
\partial_s L(t,0) = - \int_0^\infty y \nu(t,dy)\ , \quad t>0\ ,
$$
the second statement \eqref{i12} of Theorem~\ref{thint0} readily follows from Proposition~\ref{prsme0}.


\section*{Acknowledgments}

Part of this work was done while PhL enjoyed the hospitality and support of the Department of Mathematical and Statistical Sciences of the  University of Alberta, Edmonton, Canada and the African Institute for Mathematical Sciences, Muizenberg, South Africa.



\end{document}